\documentclass[12pt]{amsart}
\usepackage{amsfonts,color,morefloats,pslatex}
\usepackage{amssymb,amsthm, amsmath,latexsym}
\usepackage[pagewise]{lineno}

\newtheorem{theorem}{Theorem}
\newtheorem{lemma}[theorem]{Lemma}
\newtheorem{proposition}[theorem]{Proposition}

\newtheorem{corollary}[theorem]{Corollary}

\newcommand{\AGL}{{\mathrm{AGL}}}

\newcommand{\Aut}{{\mathrm{Aut}}} 
\newcommand{\PAut}{{\mathrm{PAut}}}

\newcommand{\lcm}{{\mathrm{lcm}}}

\newcommand{\F}{{\mathbb{F}}}

\newcommand{\M}{\mathbb{M}}

\newcommand{\cB}{{\mathcal{B}}} 
\newcommand{\N}{{\mathcal{N}}} 
\newcommand{\C}{{\mathcal{C}}}

\def\qi#1 {\fbox {\footnote {\ }}\ \footnotetext { From Qi: {\color{red}#1}}}

\begin{document}

\title[Steiner systems $S(2,4, 2^m)$ from a family of extended cyclic codes]{Steiner systems $S(2,4, 2^m)$ supported by a family of extended cyclic codes}

\author{Qi Wang}
\address{Department of Computer Science and Engineering, National Center for Applied Mathematics Shenzhen, Southern University of Science and Technology, Nanshan District, Shenzhen, Guangdong 518055, China}
\email{wangqi@sustech.edu.cn}
\thanks{}


\date{}

\dedicatory{}





\maketitle

\begin{abstract} 
In [C. Ding, An infinite family of Steiner systems $S(2,4,2^m)$ from cyclic codes, {\em J. Combin. Des.} 26 (2018), no.3, 126--144], Ding constructed a family of Steiner systems $S(2,4,2^m)$ for all $m \equiv 2 \pmod{4} \ge 6$ from a family of extended cyclic codes. The objective of this paper is to present a family of Steiner systems $S(2,4,2^m)$ for all $m \equiv 0 \pmod{4} \ge 4$ supported by this family of extended cyclic codes. The main result of this paper complements the
previous work of Ding, and the results in the two papers will show that there exists a binary extended cyclic code that can support a Steiner system $S(2,4,2^m)$ for all even $m \geq 4$. Furthermore, this paper also determines the parameters of other $2$-designs supported by this family of extended cyclic codes.         
\end{abstract}

\section{Introduction}\label{sec-intro}

A (finite) {\em incidence structure} is a triple $(V, \cB, I)$ such that $V$ is a finite set of $v$ elements called {\em points}, $\cB$ is a set of $k$-subsets of $V$ with $1 \le k \le v$, and $I \subseteq V \times \cB$ is an incidence relation between $V$ and $\cB$. The incidence relation $I$ is given by membership, i.e., a point $p \in V$ and a block $B \in \cB$ are incident if and only if $p \in B$. In the following, we may simply denote the incidence structure $(V, \cB, I)$ by a
pair $(V, \cB)$. A $t$-$(v,k,\lambda)$ {\em design} (or $t$-design, for short) with $t \leq k$ is an incidence structure $(V,\cB)$ such that any $t$-subset of $V$ is contained in exactly $\lambda$ blocks~\cite{BJL99}. We denote the number of blocks of a $t$-design by $b$, and a $t$-design is called {\em symmetric} if $v = b$. A $t$-design that has no repeated blocks is called {\em simple}. Throughout this
paper, we consider only simple $t$-designs with $t < k < v$. By a simple counting argument, we have the following identity, which restricts the parameters of a $t$-$(v,k,\lambda)$ design. 
\begin{equation}\label{eqn-para}
  b {k \choose t} = \lambda {v \choose t}.
\end{equation}
In particular, a $t$-$(v,k,\lambda)$ design is referred to as a {\em Steiner system} if $t \ge 2$ and $\lambda = 1$, and is traditionally denoted by $S(t,k,v)$.    

In the literature, the most studied Steiner systems are those with parameters $S(2,3,v)$ ({\em Steiner triple systems}), $S(3,4,v)$ ({\em Steiner quadruple systems}), and $S(2,4,v)$. It was proved in~\cite{Hana61} that a Steiner system $S(2,4,v)$ exists if and only if $v \equiv 1$ or $4 \pmod{12}$. Infinite families of Steiner systems $S(2,4,v)$ include: $S(2,4,4^n)$, where $n \ge 2$ (affine geometries); $S(2,4,3^n+\cdots +3 + 1)$, where $n \ge 2$ (projective geometries); $S(2,4,2^{s+2}
- 2^s + 4)$, where $s > 2$ (Denniston designs). For surveys on such Steiner systems, the reader is referred to~\cite{CM06,RR10}. It is open if the Steiner systems $S(2,4,v)$ listed above are supported by a linear code. We answer this problem affirmatively by investigating the Steiner systems $S(2,4,2^m)$ supported by a family of extended cyclic code for all even $m \ge 4$.  

Recently, Ding constructed an infinite family of Steiner systems $S(2,4,2^m)$ 
for all $m \equiv 2 \pmod{4} \geq 6$ from a family of extended cyclic codes, which 
are affine-invariant~\cite{Ding18}. As a follow-up of~\cite{Ding18}, in this paper, we prove that a  
family of extended binary cyclic codes hold Steiner systems $S(2,4,2^m)$ for all 
$m \equiv 0 \pmod{4} \geq 4$, and also determine the parameters of other $2$-designs supported 
by this family of extended cyclic codes. The result of this paper together with the work 
of \cite{Ding18} shows that there is an extended binary cyclic code that can support a 
Steiner system $S(2,4,2^m)$ for all even $m \geq 4$. Hence, the work of this paper 
complements that of \cite{Ding18}.

The remainder of the paper is organized as follows. In Section~\ref{sec-pre}, we will 
recall linear codes, and introduce two general methods in constructing designs from linear codes, the latter of which will be employed in Section~\ref{sec-main}. After recalling a type of affine-invariant codes, we will determine the weight distribution of their dual codes, and then explicitly compute the parameters of the underlying Steiner systems from these codes. In Section~\ref{sec-con}, we will conclude the paper and present several open problems.

\section{Preliminaries}\label{sec-pre}

In this section, we first introduce some basics of linear codes and cyclic codes, and then summarize two generic constructions of designs from linear codes.  

\subsection{Linear codes and cyclic codes}\label{sec-codes}

An $[n,\kappa,d]$ linear code $\C$ over the finite field $\F_q$ is a $\kappa$-dimensional subspace of $\F_q^n$ with minimum nonzero Hamming weight $d$, where $q$ is a power of a prime $p$. As a linear subspace of $\F_q^n$, the linear code $\C$ can be represented as the span of a minimal set of codewords. These codewords are collated as the rows of a matrix $G$, also known as a {\em generator matrix} of the code $\C$. The matrix $H$, whose null space is $\C$, is called the {\em
parity-check} matrix of $\C$. The {\em dual code} of $\C$ is the code whose generator matrix is the parity-check matrix $H$ of $\C$, denoted by $\C^\perp$, and the {\em extended code} is the code having one extra parity element, denoted by $\overline{\C}$. A linear code $\C$ over $\F_q$ is called {\em cyclic} if each codeword $(c_0, c_1, \ldots, c_{n-1})$ in $\C$ implies that the cyclic shift $(c_{n-1}, c_0, c_1, \ldots, c_{n-2})$ is also a codeword in $\C$. We may identify each vector $(c_0, c_1, \ldots, c_{n-1}) \in \F_q^n$ with 
$$
c_0 + c_1 x + c_2 x^2 + \cdots + c_{n-1} x^{n-1} \in \F_q[x] / \langle x^n - 1 \rangle,
$$
then each code $\C$ of length $n$ over $\F_q$ corresponds to a subset of $\F_q[x]/\langle x^n - 1 \rangle$. It is well-known that the linear code $\C$ is cyclic if and only if the corresponding subset in $\F_q[x]/\langle x^n -1 \rangle$ is an ideal of the ring $\F_q[x]/\langle x^n - 1 \rangle $. Let $\C = \langle g(x) \rangle$, where the generator $g(x)$ of the ideal has the least degree. Here, $g(x)$ is called a {\em generator polynomial} and $h(x) = (x^n -1) / g(x)$ is referred to as a {\em parity-check polynomial} of
$\C$. A cyclic code of length $n = q^m - 1$ over $\F_q$ for some positive integer $m$ is called a {\em primitive cyclic code}.

Let $\N = \{0,1,2, \ldots, n-1\}$. Let $s$ be an integer with $0 \leq s < n$. The {\em $q$-cyclotomic coset of $s$ modulo $n$} is defined by 
$$
C_s := \{s, sq, sq^2, \ldots, sq^{\ell_s - 1} \} \pmod{n} \subseteq \N,
$$
where $\ell_s$ is the smallest positive integer such that $s \equiv s q^{\ell_s} \pmod{n}$, and is exactly the size of the $q$-cyclotomic coset. The smallest integer in $C_s$ is called the {\em coset leader} of $C_s$. Let $\Gamma_{(n,q)}$ denote the set of all coset leaders. It is easy to see that all distinct $q$-cyclotomic cosets modulo $n$ partition $\N$, i.e., $\cup_{s \in \Gamma_{(n,q)}} C_s = \N$. The generator polynomial of the cyclic code $\C$ of length $n$ can be expressed as 
$$
g(x) = \prod_{t \in T} (x - \alpha^t),
$$
where $\alpha$ is a generator of $\F_{q^m}^*$, $T$ is a subset of $\N$ and a union of some $q$-cyclotomic cosets modulo $n$. The set $T$ is called a {\em defining set} of $\C$ with respect to $\alpha$. 

For a $[v,\kappa,d]$ linear code $\C$ over $\F_q$, let $A_i := A_i(\C)$ denote the number of codewords with Hamming weight $i$ in $\C$, where $0 \leq i \leq v$. Then the sequence $(A_0, A_1, \ldots, A_v)$ is called the {\em weight distribution} of $\C$, and $\sum_{i=0}^v A_i z^i$ is referred to as the {\em weight enumerator} of $\C$. The weight distribution is quite important 
not only for measuring the error-correcting capability of a code, but also for determining the parameters of designs supported by linear codes.   

The following lemma, a variant of the MacWilliams Identity~\cite{Lint99}, reveals how to compute the weight enumerator of the dual code $\C^\perp$ from that of the original code $\C$.  

\begin{lemma}\cite[p. 41]{Lint99}\label{lem-MI}
Let $\C$ be a $[v, \kappa, d]$ code over $\F_q$ with weight enumerator $A(z)=\sum_{i=0}^v A_iz^i$ and let
$A^\perp(z)$ be the weight enumerator of $\C^\perp$. Then
$$A^\perp(z)=q^{-\kappa}\Big(1+(q-1)z\Big)^vA\Big(\frac {1-z} {1+(q-1)z}\Big).$$

\end{lemma}

\subsection{Designs from linear codes}\label{sec-dfc}

Let $\C$ be a linear code, and let $(A_0, A_1, \ldots, A_v)$ be its weight distribution. For each $k$ with $A_k \ne 0$, by $\cB_k$ we denote the set of supports of all codewords with Hamming weight $k$ in $\C$, where the coordinates of a codeword are indexed by $(0,1,2, \ldots, v-1)$. Let $V = \{0,1,2, \ldots, v-1\}$ be the point set. The pair $(V, \cB_k)$ might be a $t$-$(v,k,\lambda)$ design for some positive integers $\lambda$ and $t$. In this case, we call this design a 
{\em support design} of the code $\C$, and say that the code $\C$ holds a $t$-$(v,k,\lambda)$ design.  

In the literature, a lot of progress has been made in constructing designs from the codewords of a fixed weight in some linear codes (for example, see~\cite{AK92,JT91,KP95,KP03,Ton98,Ton07,TDX19,Ding19,DT20,XLW20,TD21}). There are two general methods, which specify certain sufficient conditions for constructing $t$-designs from linear codes. The first method, given by the following theorem, was developed by Assmus and Mattson~\cite{AK92,HP03}. 

\begin{theorem}[Assmus-Mattson Theorem]\cite{AK92,HP03}\label{thm-am}
Let $\C$ be a $[v,\kappa,d]$ code over $\F_q$. Let $d^\perp$ denote the minimum distance of $\C^\perp$. 
Let $w$ be the largest integer satisfying $w \leq v$ and 
$$ 
w-\left\lfloor  \frac{w+q-2}{q-1} \right\rfloor <d. 
$$ 
Define $w^\perp$ analogously using $d^\perp$. Let $(A_i)_{i=0}^v$ and $(A_i^\perp)_{i=0}^v$ denote 
the weight distribution of $\C$ and $\C^\perp$, respectively. Fix a positive integer $t$ with $t<d$, and 
let $s$ be the number of $i$ with $A_i^\perp \neq 0$ for $0 \leq i \leq v-t$. Suppose $s \leq d-t$. Then 
\begin{itemize}
\item the codewords of weight $i$ in $\C$ hold a $t$-design provided $A_i \neq 0$ and $d \leq i \leq w$, and 
\item the codewords of weight $i$ in $\C^\perp$ hold a $t$-design provided $A_i^\perp \neq 0$ and 
         $d^\perp \leq i \leq \min\{v-t, w^\perp\}$. 
\end{itemize}
\end{theorem}

We note that the key sufficient condition in the Assmus-Mattson theorem is that $s \le d - t$, where $s$ denotes the number of nonzero Hamming weights in the dual code $\C^\perp$, and $d$ denotes the minimum Hamming weight in the code $\C$. In other words, $d-s$ is the upper bound for $t$ when a linear code $\C$ holds a $t$-design. Recently, the Assmus-Mattson theorem was explicitly employed to construct infinite families of $t$-designs, see~\cite{Ding182,DL17}.

The other method is to consider the automorphism group of a linear code. 
If the automorphism group $\Aut(\C)$ of a linear code is $t$-transitive, then the codewords of weight greater than $t$ holds a $t$-design. The set of coordinate permutations that map a code $\C$ to itself forms a group. This group is called the {\em permutation automorphism group} of $\C$ and is denoted by $\PAut(\C)$. The {\em automorphism group} of a linear code $\C$, denoted by $\Aut(\C)$, is the set of maps of the form $DP\gamma$,
where $D$ is a diagonal matrix, $P$ is a permutation matrix, and $\gamma$ is an automorphism of $\F_q$, that map $\C$ to itself. Clearly, we have $\PAut(\C) \leqslant \Aut(\C)$. The automorphism group $\Aut(\C)$ is said to be {\em $t$-transitive} if for every pair of $t$-element ordered subsets of the index set $V=\{0,1,...,v-1\}$, there is an element $DP \gamma$ of the automorphism group $\Aut(\C)$, such that its permutation part $P$ sends the first set to the second set. The following theorem summarizes this approach, whose proof can be found in~\cite[p. 308]{HP03}. 

\begin{theorem}\cite{HP03}\label{thm-aut}
Let $\C$ be a linear code of length $n$ over $\F_q$ where $\Aut(\C)$ is $t$-transitive. Then the codewords of any weight $i \geq t$ of $\C$ hold a $t$-design.
\end{theorem} 

To apply this method, one has to find a subgroup of the automorphism group $\Aut(\C)$ that is 
$t$-transitive on $\C$. Note that it is in general very hard to determine the parameters of such $t$-design supported by a linear code. Some examples of $t$-designs obtained by this method can be found, for example, in~\cite{CMOT06,LTW12,DZ17,DW20,DW22}. The following result states that the dual code $\C^\perp$ also holds $t$-designs~\cite{MS771} if $\C$ does so. 

\begin{proposition}\cite{MS771}\label{prop-dualdesign}
Let $\C$ be an $[n, \kappa, d]$ binary linear code with $\kappa>1$, such that for each weight $w>0$ the supports of the codewords of weight $w$ form a $t$-design, where $t<d$. Then the supports of the codewords of each nonzero weight in $\C^\perp$ also form a $t$-design.   
\end{proposition}

\section{A type of affine-invariant codes and their designs}\label{sec-main} 

In this section, we first recall a type of affine-invariant codes investigated in~\cite{Ding18}, together with some known results therein, and then show that a family of the codes supports 
Steiner systems $S(2, 4, 2^m)$ for all $m \equiv 0 \pmod{4} \geq 4$. 

\subsection{A type of affine-invariant codes}\label{sec-maincodes}

An {\em affine-invariant} code over $\F_q$ is an extended primitive cyclic code $\overline{\C}$ of length $q^m$ such that the {\em affine permutation group}, denoted by $\AGL(1,q^m)$, is a subset of $\PAut(\C)$, where $\AGL(1,q^m)$ is defined by 
$$
\AGL(1,q^m) = \{ \sigma_{(a,b)} (y) = ay + b: \ a \in \F_{q^m}^*, \ b \in \F_{q^m} \}.
$$
Note that $\AGL(1,q^m)$ acts doubly transitively on $\F_{q^m}$. It then follows from Theorem~\ref{thm-aut} that an affine-invariant code $\C$ holds $2$-designs~\cite{Ding18}. Besides several previously known affine-invariant codes, including the extended narrow-sense primitive BCH
codes~\cite{DZ17}, the classical generalized Reed-Muller codes, and a family of newly generalized Reed-Muller codes~\cite{DLX18}, there exists another type of affine-invariant codes~\cite{Ding18}, which are recalled below.

Let $m \ge 2 $ be a positive integer. Define $\overline{m} = \lfloor m/2 \rfloor$ and $M = \{1, 2, \ldots, \overline{m}\}$. Let $E$ be any nonempty subset of $M$. Define
$$
g_E (x) = \M_\alpha (x) \lcm \{ \M_{\alpha^{1 + 2^e}}(x): \ e \in E\},
$$
where $\alpha$ is a generator of $\F_{2^m}^*$, $\M_{\alpha^i} (x)$ denotes the minimal polynomial of $\alpha^i$ over $\F_2$, and $\lcm$ denotes the least common multiple of a set of polynomials. We denote by $\C_E$ the binary cyclic code of length $n$ with the generator polynomial $g_E(x)$, where $n = 2^m - 1$. It was proved in~\cite[Theorems 13, 14]{Ding18} that the generator polynomial is given by  
\begin{equation}\label{eqn-ge}
g_E (x) = \M_\alpha(x) \prod_{e \in E} \M_{\alpha^{1 + 2^e}}(x) ,
\end{equation}
and the extended code $\overline{\C_E}$ is affine-invariant. Therefore, by Theorem~\ref{thm-aut} and Proposition~\ref{prop-dualdesign}, the supports of the codewords of each weight $k$ in $\overline{C_E}$ and $\overline{C_E}^\perp$ form a $2$-design, if $\overline{A}_k \ne 0$ and $\overline{A}_k^\perp \ne 0$, respectively. We remark here that, to construct $2$-designs from these affine-invariant codes and explicitly determine the parameters of these designs, it still remains
to determine the weight distributions of these codes. Only a few special cases of these codes are possible. Ding~\cite{Ding18} considered one special case that $E = \{ e \}$, where $1 \leq e \leq \overline{m} = \lfloor m/2 \rfloor$. For other special cases of $E$ and the designs from the corresponding codes, see~\cite{Ding182,DZ17,Kas69}, and a summary in~\cite{Ding18}. The following theorem~\cite{Ding18} determines the parameters and the weight distributions of the codes $\overline{\C_E}$ and $\overline{\C_E}^\perp$.  

\vspace{.25cm}
\begin{theorem}~\cite[Theorem 17]{Ding18}\label{thm-wd}
  Let $m \ge 4$ and $1 \leq e \le \overline{m}$. 
\begin{itemize}
  \item[(a)] When $m/\gcd(m,e)$ is odd, define $h=(m-\gcd(m,e))/2$. Then $\overline{\C_{E}}^\perp$ 
has parameters $[2^m, 2m+1, 2^{m-1}-2^{m-1-h}]$, and the weight
distribution in Table~\ref{tab-aaCG1jan18}. 
The parameters of $\overline{\C_{E}}$ are $[2^m, 2^m-1-2m, \overline{d}]$, 
where 
$$
\overline{d} = \left\{ \begin{array}{ll}
                             4 & \mbox{ if } \gcd(e,m)>1; \\
                             6 & \mbox{ if } \gcd(e,m)=1.
\end{array}
\right.
$$

\item[(b)] When $m$ is even and $e=m/2$, $\overline{\C_{E}}^\perp$ has parameters $[2^m, 1+3m/2, 2^{m-1}-2^{(m-2)/2}]$ and the weight
distribution in Table \ref{tab-aaCG2jan18}. 
The code $\overline{\C_{E}}$ has parameters $[2^m, 2^m-1-3m/2, 4]$. 

\item[(c)] When $m/\gcd(m,e)$ is even and $1 \leq e < m/2$, $\overline{\C_{E}}^\perp$ has parameters
$$[2^m, \, 2m+1,\, 2^{m-1}-2^{(m+\ell-2)/2}]$$ and the weight 
distribution in Table \ref{tab-aaMinejan18}, where $\ell=2\gcd(m, e)$, and $\overline{\C_{E}}$ has parameters $[2^m, 2^m-1-2m, \overline{d}]$, where 
$$
\overline{d} = \left\{ \begin{array}{ll}
                             4 & \mbox{ if } \gcd(e,m)>1; \\
                             6 & \mbox{ if } \gcd(e,m)=1.
\end{array}
\right.
$$
\end{itemize}
\end{theorem}

\begin{table}[ht]
\center 
\caption{Weight distribution of $\overline{\C_E}^\perp$ (a)}\label{tab-aaCG1jan18} 
{
\begin{tabular}{ll}
\hline
Weight $w$    & No. of codewords $A_w$  \\ \hline
$0$                                                        & $1$ \\ 
$2^{m-1}-2^{m-1-h}$           & $(2^m-1)2^{2h}$ \\ 
$2^{m-1}$                             & $(2^m-1)(2^{m+1}-2^{2h+1}+2)$ \\ 
$2^{m-1}+2^{m-1-h}$           & $(2^m-1)2^{2h}$ \\  
$2^m$                           &  $1$ \\ \hline
\end{tabular}
}
\end{table}

\begin{table}[ht]
\center 
\caption{Weight distribution of $\overline{\C_E}^\perp$ (b)}\label{tab-aaCG2jan18}
{
\begin{tabular}{ll}
\hline
Weight $w$    & No. of codewords $A_w$  \\ \hline
$0$                                                        & $1$ \\ 
$2^{m-1}-2^{(m-2)/2}$          & $(2^{m/2}-1)2^{m}$ \\ 
$2^{m-1}$                               & $2^{m+1}-2$ \\ 
$2^{m-1}+2^{(m-2)/2}$          & $(2^{m/2}-1)2^{m}$ \\ 
$2^m$                          & $1$ \\ \hline
\end{tabular}
}
\end{table}

\begin{table}[ht]
\center 
\caption{Weight distribution of $\overline{\C_E}^\perp$ (c)}\label{tab-aaMinejan18}
{
\begin{tabular}{ll}
\hline
Weight $w$    & No. of codewords $A_w$  \\ \hline
$0$                                                        & $1$ \\ 
$2^{m-1}-2^{(m+\ell-2)/2}$          & $2^{m-\ell} (2^m-1)/(2^{\ell/2}+1)$ \\ 
$2^{m-1}-2^{(m-2)/2}$          & $2^{(2m+\ell)/2} (2^m-1)/(2^{\ell/2}+1)$ \\ 
$2^{m-1}$                               & $2( (2^{\ell/2} -1)2^{m-\ell} +1 )(2^m - 1)$ \\ 
$2^{m-1}+2^{(m-2)/2}$          & $2^{(2m+\ell)/2} (2^m-1)/(2^{\ell/2}+1)$ \\ 
$2^{m-1}+2^{(m+\ell-2)/2}$          & $2^{m-\ell} (2^m-1)/(2^{\ell/2}+1)$ \\ 
$2^m$                               & $1$ \\ \hline 
\end{tabular}
}
\end{table}

\subsection{Designs from these affine-invariant codes}\label{sec-maindesigns}

With Theorem~\ref{thm-wd}, it is possible to formulate the parameters of the $2$-designs from the codes $\overline{\C_E}^\perp$ by (\ref{eqn-para}) in the following theorem~\cite[Theorem 18]{Ding18}. 

\vspace{.25cm}
\begin{theorem}~\cite[Theorem 18]{Ding18}\label{thm-designdual}
  Let $m \ge 4$ and $1 \leq e \le \overline{m}$. 
\begin{itemize}
    
\item[(a)] When $m/\gcd(m,e)$ is odd, define $h=(m-\gcd(m,e))/2$. Then $\overline{\C_{E}}^\perp$ 
holds a $2$-$(2^m, k, \lambda)$ design for the following pairs $(k, \lambda)$: 
\begin{itemize}
\item $(k, \lambda)=\left(2^{m-1} \pm 2^{m-1-h},\  (2^{2h-1}\pm 2^{h-1})(2^{m-1} \pm 2^{m-1-h}-1)\right)$.  
\item $(k, \lambda)=\left(2^{m-1}, \  (2^{m-1}-1)(2^{m}-2^{2h}+1) \right).$ 
\end{itemize}

\item[(b)] When $m$ is even and $e=m/2$, $\overline{\C_{E}}^\perp$ 
holds a $2$-$(2^m, k, \lambda)$ design for the following pairs $(k, \lambda)$: 
\begin{itemize}
\item $(k, \lambda)=\left(2^{m-1} \pm 2^{(m-2)/2}, \ 
       2^{(m-2)/2}(2^{m/2} - 1)(2^{(m-2)/2} \pm 1) \right).$ 
\item $(k, \lambda)=\left(2^{m-1}, \ 2^{m-1}-1\right).$ 
\end{itemize}

\item[(c)] When $m/\gcd(m,e)$ is even and $1 \leq e < m/2$, $\overline{\C_{E}}^\perp$ 
holds a $2$-$(2^m, k, \lambda)$ design for the following pairs $(k, \lambda)$: 
\begin{itemize}
\item $(k, \lambda)=\left(2^{m-1} \pm 2^{(m+\ell-2)/2}, \ 
            \frac{(2^{m-1} \pm 2^{(m+\ell-2)/2})(2^{m-1} \pm 2^{(m+\ell-2)/2}-1)}{2^\ell(2^{\ell/2}+1)} \right),$ 
\item $(k, \lambda)=\left(2^{m-1} \pm 2^{(m-2)/2}, \ 
       \frac{2^{(m+\ell-2)/2}(2^{m/2} \pm 1)(2^{m-1} \pm 2^{(m-2)/2}-1)}{2^{\ell/2}+1} \right),$ 
\item $(k, \lambda)=\left(2^{m-1}, \ ( (2^{\ell/2} -1)2^{m-\ell} +1 )(2^{m-1} - 1) \right),$ 
\end{itemize}  
where $\ell=2\gcd(m, e)$. 
\end{itemize}
\end{theorem}

In~\cite{Ding18}, a special case ($2 \leq e \leq \overline{m}$ and $\gcd(m,e) = 2$) was considered to determine the weight distribution of the code $\overline{\C_E}$, and then Steiner systems $S(2,4,2^m)$ for $m \equiv 2 \pmod{4}$ were constructed. 

\begin{theorem}\cite[Theorem 21]{Ding18}\label{thm-main0}
Let $m \equiv 2 \pmod{4} \ge 4$, $2 \leq e \leq \lfloor m/2 \rfloor$, and $\gcd(m, e)=2$. Then 
the supports of the codewords of weight $4$ in $\overline{\C_E}$ form a $2$-$(2^m, 4, 1)$ 
design, i.e., a Steiner system $S(2, 4, 2^m)$. 
\end{theorem}

We now deal with the special case for $m \equiv 0 \pmod{4}$. More precisely, we first determine the weight distribution of the code $\overline{\C_E}$ from that of the code $\overline{\C_E}^\perp$. This could be done by giving a corollary of Theorem~\ref{thm-wd}, and then deriving the weight distribution of the code $\overline{\C_E}$ by Lemma~\ref{lem-MI}. 

\begin{table}[ht]
\center 
\caption{Weight distribution of $\overline{\C_E}^\perp$}\label{tab-aaMinejan18WQ}
{
\begin{tabular}{ll}
\hline
Weight $w$    & No. of codewords $A_w$  \\ \hline
$0$                                                        & $1$ \\ 
$2^{m-1}-2^{(m+2)/2}$          & $2^{m-4} (2^m-1)/5$ \\ 
$2^{m-1}-2^{(m-2)/2}$          & $2^{m+2} (2^m-1)/5$ \\ 
$2^{m-1}$                               & $2( 3 \cdot 2^{m-4} +1 )(2^m - 1)$ \\ 
$2^{m-1}+2^{(m-2)/2}$          & $2^{m+2} (2^m-1)/5$ \\ 
$2^{m-1}+2^{(m+2)/2}$          & $2^{m-4} (2^m-1)/5$ \\ 
$2^m$                               & $1$ \\ \hline 
\end{tabular}
}
\end{table}

\begin{corollary}\label{coro-wddual}
Let $m \equiv 0 \pmod{4} \ge 4$ and $1 \leq e \le m/2$.  When $\gcd(m, e)=2$, 
$\overline{\C_{E}}^\perp$ has parameters
$$[2^m, \, 2m+1,\, 2^{m-1}-2^{(m+2)/2}]$$ and the weight 
distribution in Table \ref{tab-aaMinejan18WQ}, and $\overline{\C_{E}}$ has parameters $[2^m, 2^m-1-2m, 4]$.
\end{corollary}

With the help of Lemma~\ref{lem-MI}, we are able to determine the weight distribution of $\overline{\C_E}$, based on the weight distribution of $\overline{\C_E}^\perp$ given in Corollary~\ref{coro-wddual}.

\begin{lemma}\label{lem-wdext}
Let $m \equiv 0 \pmod{4} \ge 4$ and $1 \leq e \le m/2$.  When $\gcd(m, e)=2$, $\overline{\C_{E}}$ has parameters $[2^m, 2^m-1-2m, 4]$, and its weight distribution is given by 
\begin{eqnarray*}
2^{2m+1}\overline{A}_k =  
  \left(1+(-1)^k \right) \binom{2^m}{k} + w E_0(k) + u E_1(k) +  v E_2(k), 
\end{eqnarray*} 
where  
\begin{eqnarray*}
u &=& 2^{m-4} (2^m-1)/5, \\ 
v &=& 2^{m+2} (2^m-1)/5, \\ 
w &=& 2( 3 \cdot 2^{m-4} +1 )(2^m - 1),  
\end{eqnarray*} 
and 
\begin{eqnarray*}
 E_0(k) = \frac{1+(-1)^k}{2} (-1)^{\lfloor k/2 \rfloor} \binom{2^{m-1}}{\lfloor k/2 \rfloor}, 
\end{eqnarray*} 
\begin{eqnarray*}
 E_1(k) = \sum_{\substack{0 \le i \le 2^{m-1}-2^{(m+2)/2} \\
0\le j \le 2^{m-1}+2^{(m+2)/2} \\i+j=k}}((-1)^i +(-1)^j) \binom{2^{m-1}-2^{(m+2)/2}} {i} \binom{2^{m-1}+2^{(m+2)/2}}{j},  
\end{eqnarray*}
\begin{eqnarray*}
 E_2(k) =  \sum_{\substack{0 \le i \le 2^{m-1}-2^{(m-2)/2} \\
0\le j \le 2^{m-1}+2^{(m-2)/2} \\i+j=k}} ((-1)^i + (-1)^j) \binom{2^{m-1}-2^{(m-2)/2}} {i} \binom{2^{m-1}+2^{(m-2)/2}}{j},   
\end{eqnarray*} 
and $0 \leq k \leq 2^m$.  
\end{lemma} 

\begin{proof}
  With the weight distribution of $\overline{\C_E}^\perp$ given in Table \ref{tab-aaMinejan18WQ}, the weight enumerator $\overline{A}^\perp (z)$ of $\overline{\C_E}^\perp$ is the following
  \begin{eqnarray*}
    \overline{A}^\perp (z) & = & 1 + u z^{2^{m-1} - 2^{(m+2)/2}} + v z^{2^{m-1} - 2^{(m-2)/2}} + w z^{2^{m-1}} \\
    &  & + v z^{2^{m-1} + 2^{(m-2)/2}} + u z^{2^{m-1} + 2^{(m+2)/2}} + z^{2^m} ,
  \end{eqnarray*}
  where $u = 2^{m-4} (2^m - 1) /5$, $v = 2^{m+2} (2^m  - 1)/5$, and $w = 2(3 \cdot 2^{m-4} + 1)(2^m - 1)$. By Lemma~\ref{lem-MI}, in the following we are able to give the weight enumerator $\overline{A}(z)$ of the code $\overline{\C_E}$, which is the dual of the code $\overline{\C_E}^\perp$.
  \begin{eqnarray}\label{eqn-wtenum}
    \lefteqn{ 2^{2m+1} \overline{A}(z) } \nonumber\\
    & = & (1 + z)^{2^m} \overline{A}^\perp \left( \frac{1-z}{1+z} \right) \nonumber \\
    & = & (1+z)^{2^m} + u (1-z)^{2^{m-1} - 2^{(m+2)/2}} (1+z)^{2^{m-1} + 2^{(m+2)/2}} \nonumber \\
    &  & + v (1-z)^{2^{m-1} - 2^{(m-2)/2}} (1+z)^{2^{m-1} + 2^{(m-2)/2}} + w (1-z)^{2^{m-1}} (1+z)^{2^{m-1}} \nonumber \\
    &   &   +  v (1-z)^{2^{m-1} + 2^{(m-2)/2}} (1+z)^{2^{m-1} - 2^{(m-2)/2}} \nonumber \\
    &  &  + u (1-z)^{2^{m-1} + 2^{(m+2)/2}} (1+z)^{2^{m-1} - 2^{(m+2)/2}} + (1-z)^{2^m}. 
  \end{eqnarray}
  We now explicitly compute from the equation above the number $\overline{A}_k$ of codewords of weight $k$ in the code $\overline{\C_E}$. From the first and the last terms in (\ref{eqn-wtenum}), the term 
  $$
  \left(1+(-1)^k \right) \binom{2^m}{k}
  $$
  should be added in $2^{2m+1} \overline{A}_k$. We then count the number of codewords of weight $k$ from the two terms in (\ref{eqn-wtenum}) with the same coefficient $u$. The first gives 
\begin{eqnarray*}
\sum_{\substack{0 \le i \le 2^{m-1}-2^{(m+2)/2} \\
0\le j \le 2^{m-1}+2^{(m+2)/2} \\i+j=k}} (-1)^i \binom{2^{m-1}-2^{(m+2)/2}} {i} \binom{2^{m-1}+2^{(m+2)/2}}{j},  
\end{eqnarray*}
and the second term gives
\begin{eqnarray*}
\sum_{\substack{0 \le i' \le 2^{m-1}-2^{(m+2)/2} \\
0\le j' \le 2^{m-1}+2^{(m+2)/2} \\i'+j'=k}} (-1)^{j'} \binom{2^{m-1}-2^{(m+2)/2}} {i'} \binom{2^{m-1}+2^{(m+2)/2}}{j'}.  
\end{eqnarray*}
These two sum up to $E_1(k)$, as stated in the present lemma. The remaining numbers of codewords of weight $k$ can be computed similarly, and so the detail is omitted. The proof is then completed.   

\end{proof}

By Lemma~\ref{lem-wdext}, the number of codewords of Hamming weight $k$ in $\overline{\C_E}$ can be accordingly determined in the following lemma. 

\begin{lemma}\label{lem-numwt}
Let $m \equiv 0 \pmod{4} \ge 4$ and $1 \leq e \le m/2$.  When $\gcd(m, e)=2$, $\overline{\C_{E}}$ has parameters $[2^m, 2^m-1-2m, 4]$, and 
\begin{eqnarray*}
\overline{A}_4= \frac{2^m(2^m-1)}{12}, 
\end{eqnarray*} 
\begin{eqnarray*}
\overline{A}_6=\frac{2^m(2^m-1)(2^{2m-4}+2^{m-1}+6)}{45}, 
\end{eqnarray*} 
\begin{eqnarray*}
\overline{A}_8=\frac{2^{m-3}(2^m-1)(2^{4m-4}-27\cdot 2^{3m-4}+23\cdot 2^{2m-1} +261\cdot 2^{m-2}+403)}{315}.  
\end{eqnarray*}
\end{lemma}

Now we give the main theorem, which confirms the parameters of the Steiner systems from the affine-invariant codes $\overline{\C_E}$, where $E = \{ e \}$. 

\begin{theorem}\label{thm-main}
Let $m \equiv 0 \pmod{4} \ge 4$, $2 \leq e \leq  m/2 $, and $\gcd(m, e)=2$. Then 
the supports of the codewords of weight $4$ in $\overline{\C_E}$ form a $2$-$(2^m, 4, 1)$ 
design, i.e., a Steiner system $S(2, 4, 2^m)$. 
\end{theorem} 

\begin{proof}
Using the weight distribution formula $\overline{A}_k$ given in Lemma \ref{lem-numwt}, 
we obtain 
$$ 
\overline{A}_4= \frac{2^{m-1}(2^m-1)}{6}. 
$$ 
It then follows from (\ref{eqn-para}) that 
$$ 
\lambda=\overline{A}_4 \frac{\binom{4}{2}}{\binom{2^m}{2}}=1. 
$$ 
This completes the proof. 
\end{proof}

Combined with Theorem~\ref{thm-main0}, we obtain Steiner systems $S(2, 4, 2^m)$ for all even $m \ge 4$. 

\begin{theorem}\label{thm-combined}
  Let $m \geq 4$ be even, $2 \leq e \leq  m/2 $, and $\gcd(m, e)=2$. Then the supports of the codewords of weight $4$ in $\overline{\C_E}$ form a $2$-$(2^m, 4, 1)$ 
design, i.e., a Steiner system $S(2, 4, 2^m)$.
\end{theorem}  

We remark that for every $m \ge 4$ even, there are flexible choices of $e$ such that $\gcd(m, e) = 2$. As a special case, $e = 2$ always satisfies the condition and Theorem~\ref{thm-combined} will give infinite families of Steiner systems $S(2,4,2^m)$.


As a byproduct, $2$-designs could be constructed from codewords of Hamming weight $k > 4$ of the code $\overline{\C_E}$. Two examples of $2$-designs with parameters are presented in the following. 

\begin{theorem}\label{thm-wt6des}
Let $m \equiv 0 \pmod{4} \ge 4$, $2 \leq e \leq \lfloor m/2 \rfloor$, and $\gcd(m, e)=2$. Then 
the supports of the codewords of weight $6$ in $\overline{\C_E}$ form a $2$-$(2^m, 6, \lambda)$ 
design, where 
$$ 
\lambda = \frac{2^{2m-3} + 2^m + 12}{3}. 
$$ 
\end{theorem} 

\begin{proof}
Using the weight distribution formula $\overline{A}_6$ given in Lemma~\ref{lem-numwt}, 
\begin{eqnarray*}
\overline{A}_6=\frac{2^m(2^m-1)(2^{2m-4}+2^{m-1}+6)}{45}, 
\end{eqnarray*} 
from which it follows that 
$$ 
\lambda=\overline{A}_6 \frac{\binom{6}{2}}{\binom{2^m}{2}}=\frac{2^{2m-3} + 2^m + 12}{3}. 
$$ 
This completes the proof. 
\end{proof}

\begin{theorem}\label{thm-wt8des}
Let $m \equiv 0 \pmod{4} \ge 4$, $2 \leq e \leq \lfloor m/2 \rfloor$, and $\gcd(m, e)=2$. Then 
the supports of the codewords of weight $8$ in $\overline{\C_E}$ form a $2$-$(2^m, 8, \lambda)$ 
design, where 
$$ 
\lambda=\frac{2^{4m-4}-27\cdot 2^{3m-4}+23\cdot 2^{2m-1} +261\cdot 2^{m-2}+403}{45}. 
$$ 
\end{theorem} 

\begin{proof}
  Using the weight distribution formula $\overline{A}_8$ given in Lemma~\ref{lem-numwt}, 
\begin{eqnarray*}
\overline{A}_8=\frac{2^{m-3}(2^m-1)(2^{4m-4}-27\cdot 2^{3m-4}+23\cdot 2^{2m-1} +261\cdot 2^{m-2}+403)}{315}.  
\end{eqnarray*}
It then follows that 
$$ 
\lambda=\overline{A}_8 \frac{\binom{8}{2}}{\binom{2^m}{2}}= \frac{2^{4m-4}-27\cdot 2^{3m-4}+23\cdot 2^{2m-1} +261\cdot 2^{m-2}+403}{45}. 
$$ 
The proof is then completed. 
\end{proof}

\subsection{Discussion on isomorphism}\label{sec-iso}

As stated in Theorem~\ref{thm-combined}, the supports of all the codewords with Hamming weight $4$ in the code $\overline{\mathcal{C}_E}$ form a Steiner system $S(2,4,2^m)$, where $m \ge 4$ is even. The points and lines in the affine geometries are the classical Steiner systems with the same parameters. Indeed, the Steiner systems discussed in this paper are in fact isomorphic to those from affine geometries. To see this, we denote by $\alpha$ a primitive element in $\mathbb{F}_4$ and
let $m' = m/2$ for even $m \ge 4$. Then we consider two sets $\mathcal{B}_4$ and $PL$, where
$\mathcal{B}_4$ denotes the set of supports of all codewords with Hamming weight $4$ in the code $\overline{\mathcal{C}_E}$, and $PL$ is defined as follows: 
\begin{eqnarray*}
  PL & = & \{ \{x, x+(x+y), x + \alpha (x+y), x+\alpha^2 (x+y) \}: \ x, y \in \mathbb{F}_{4^{m'}}, \ x \ne y \} \\
  & = & \{ \{x, y, \alpha^2 x + \alpha y, \alpha x + \alpha^2 y\} : \ x, y \in \mathbb{F}_{4^{m'}}, \ x \ne y \}.
\end{eqnarray*}
In fact, the set $PL$ is exactly the set of lines in the affine space $\mathrm{AG}(m', \mathbb{F}_4)$, and is thereby the set of blocks in the classical Steiner system $S(2,4,4^{m'})$. We now prove that $PL = \mathcal{B}_4$. On the one hand, we have 
$$
|PL| = \frac{ {4^{m'} \choose 2}}{{4 \choose 2}} = \frac{2^{m-1}(2^m -1)}{6} = |\mathcal{B}_4|. 
$$
On the other hand, it is checked that 
$$
\mathrm{Tr}(x) + \mathrm{Tr}(y) + \mathrm{Tr}(\alpha x + \alpha^2 y) + \mathrm{Tr}(\alpha^2 x + \alpha y) = 0, 
$$
and 
$$
\mathrm{Tr}(x^{2^e + 1}) + \mathrm{Tr}(y^{2^e + 1}) + \mathrm{Tr}( (\alpha x + \alpha^2 y)^{2^e + 1}) + \mathrm{Tr}( (\alpha^2 x + \alpha y)^{2^e +1}) = 0. 
$$
Then by the definition of the generator polynomial of the code $\mathcal{C}_E$ in (\ref{eqn-ge}), it follows that the set $PL$ is a subset of $\mathcal{B}_4$. Therefore, the two sets $\mathcal{B}_4$ and $PL$ are identical, and the isomorphism is thereby proved.


\section{Concluding remarks}\label{sec-con} 

The contribution of this paper is to prove that a family of extended cyclic 
codes supports Steiner systems $S(2,4,2^m)$ for all $m \equiv 0 \pmod{4} \geq 4$. The result of this paper is a follow-up of~\cite{Ding18}, and also complements that of \cite{Ding18}. Although the Steiner systems $S(2,4,2^m)$ presented in this paper and~\cite{Ding18} are isomorphic to those of points and lines in the affine geometries, the contributions of this paper and~\cite{Ding18} discovered a coding-theoretic construction of the classical geometric Steiner systems.


The key motivation of this paper comes from the following challenging research problems: 
\begin{itemize}
\item Is a given $t$-design a support design of a linear code? In other words, 
      is there a linear code that holds a given $t$-design? 
\item Construct a specific linear code that holds a given $t$-design. 
\item Construct a specific linear code with minimum dimension that holds a given $t$-design.       
\item Construct new $t$-designs from known or new linear codes. 
\end{itemize} 
It would be nice if a coding-theoretic construction of known Steiner systems can be worked out. 
The reader is cordially invited to attack this problem.

\section*{Acknowledgements} 

The author is very grateful to the two anonymous reviewers for their comments and suggestions that much improved 
the quality of this paper. In particular, it is much appreciated that one reviewer pointed out the isomorphism between the Steiner systems discussed in this paper and those from affine geometries with the same parameters. 




\providecommand{\href}[2]{#2}
\providecommand{\arxiv}[1]{\href{http://arxiv.org/abs/#1}{arXiv:#1}}
\providecommand{\url}[1]{\texttt{#1}}
\providecommand{\urlprefix}{URL }

\end{document}